\documentclass[12pt,oneside,reqno,bibliography=totoc]{scrartcl}

\usepackage{cihan}

\author{Cihan Bahran}
\affil{Department of Mathematics, Bilkent University\\
06800 Ankara, Turkey}

\date{}

\AtBeginDocument{%
  \addtolength\abovedisplayskip{-0.5\baselineskip}%
  \addtolength\belowdisplayskip{-0.5\baselineskip}%
}

\usepackage{appendix}

\title{Local degrees of $\FI$-modules through spectral sequences}

\newcommand{\FI}{\mathbf{FI}}

\newcounter{dummy}
\makeatletter
\newcommand\sitem[1][]{\item[(#1)]\refstepcounter{dummy}\def\@currentlabel{#1}}
\makeatother

\newcommand{\cofi}[1]{\co_{#1}^{\FI}}

\DeclareMathOperator{\reg}{reg}

\DeclareMathOperator{\kk}{\mathbb{F}}

\newcommand{\tgen}{t_{0}}
\newcommand{\trel}{t_{1}}
\newcommand{\shift}[2]{\mathbf{\Sigma}^{#2}   #1 }

\newcommand{\weak}{\delta}

\newcommand{\deriv}{\mathbf{\Delta}}

\newcommand{\hell}{h}
\newcommand{\local}{\hell^{\text{max}}}

\newcommand{\locoh}[1]{\co_{\mathfrak{m}}^{#1}}

\AtBeginDocument{%
   \def\MR#1{}
}

\makeatletter
\def\blfootnote{\gdef\@thefnmark{}\@footnotetext}
\makeatother

\begin{document}
\maketitle

\begin{onecolabstract} 
We establish bounds on the local degrees of $\FI$-modules (in the sense of Church--Miller--Nagpal--Reinhold \cite{cmnr-range}) through a spectral sequence which only depend on the initial page and is independent from how many pages have been flipped.
\end{onecolabstract}
{\tableofcontents}

\section{Introduction} 
The stable degree $\weak(V)$ and local degree $\local(V)$ of an $\FI$-module\footnote{Recall that an $\FI$-module $V$ is a functor $V \colon \FI \to \lMod{\zz}$ where $\FI$ is the category of finite sets and injections. We write $V_{S}$ for its evaluation at a finite set $S$ and $V_{f} \colon V_{S} \to V_{T}$ for the induced map by an injection $f \colon S \emb T$. The functor category $[\FI, \lMod{\zz}]$ is denoted $\lMod{\FI}$.} $V$ were studied in \cite{cmnr-range}. They have the following properties:
\begin{birki}
 \item Concrete manifestation in the growth behavior of $V$: if $V$ is defined over a field $\kk$ with finite-dimensional evaluations and $\weak(V) < \infty$, the dimension function
 \[
  n \mapsto \dim_{\kk} V_{n}
 \]
is equal to a polynomial (in $n$) of degree $\weak(V)$ in the range $n > \local(V)$ \cite[Proposition 2.14]{cmnr-range}.
\vspace{0.1cm}
 \item Theoretical flexibility: the invariants $\weak$ and $\local$  can be propagated through kernels and cokernels \cite[Proposition 3.3]{cmnr-range} in a more efficient way than some of the other invariants of $\FI$-modules. 
\end{birki}
Using (2) above, \cite{cmnr-range} obtain a general result (see Theorem \ref{eski-range}) that they call a ``type A spectral sequence argument'' about how $\weak$ and $\local$ change as one flips through the pages of a spectral sequence of $\FI$-modules. A drawback of their result is that the bounds tend to get progressively worse as one flips to further pages. The main objective of this paper is to use (2) more carefully to establish \textbf{page independent} bounds in this setting.

\paragraph{Degree and torsion.} Given an $\FI$-module $W$, we write 
\begin{align*}
 \deg(W) &:= \min\{d \geq -1 : W_{S} = 0 \text{ for } |S| > d\} 
 \\
 &\in \{-1,0,1,2,3,\dots\} \cup \{\infty\} \, .
\end{align*}
An $\FI$-module $V$ is \textbf{torsion} if for every finite set $S$ and $x \in V_{S}$, there exists an injection $\alpha \colon S \emb T$ such that $V_{\alpha}(x) = 0 \in V_{T}$. We write
\begin{align*}
 \locoh{0} \colon \lMod{\FI} \rarr \lMod{\FI}
\end{align*}
for the functor which assigns an $\FI$-module its largest torsion $\FI$-submodule, and write 
\begin{align*}
 h^{0}(V) := \deg(\locoh{0}(V)) \, .
\end{align*}

\paragraph{Local cohomology and local degree.}
 The functor $\locoh{0}$ is left exact. For each $j \geq 0$, we write $\locoh{j} := \operatorname{R}^{j}\!\locoh{0}$ for the $j$-th right derived functor of $\locoh{0}$, and write 
\begin{align*}
 h^{j}(V) &\coloneqq \deg(\locoh{j}(V)) 
 \in \{-1,0,1,\dots\} \cup \{\infty\} \, ,
 \\
 \local(V) &\coloneqq \max\{h^{j}(V) : j \geq 0\}
 \in \{-1,0,1,\dots\} \cup \{\infty\} \, ,
\end{align*}
for every $\FI$-module $V$. We call $\local(V)$ the \textbf{local degree} of $V$.

\paragraph{Stable degree.} Given any $\FI$-module $V$, we write $\shift{V}{}$ for the composition 
\begin{align*}
 \FI \xrightarrow{- \sqcup \{*\}} \FI \xrightarrow{V} \lMod{\zz} \, .
\end{align*}
The assignment $V \mapsto \shift{V}{}$ defines a functor $\shift{}{} \colon \lMod{\FI} \to \lMod{\FI}$ which receives a natural transformation from the identity functor $\id_{\lMod{\FI}}$; the cokernel is denoted
\begin{align*}
 \deriv := \coker \left( \id_{\lMod{\FI}} \rarr \shift{}{} \right) \, .
\end{align*}
We set
\begin{align*}
 \weak(V) &:= \min\{r \geq -1 : \deriv^{r+1}(V) \text{ is torsion}\}
 \\
 &\in \{-1,0,1,\dots\} \cup \{\infty\} \, ,
\end{align*}
and call it the \textbf{stable degree} of $V$.

We shall call a sequence of cochain complexes $\left(V^{\star}_{r} : r \geq r_{0}\right)$ with the initial page $r_{0} \in \zz$ in an abelian category a (cohomological) \textbf{single-graded spectral sequence} if 
\[\co^{k}\!\big( V^{\star}_{r} \big) \cong V^{k}_{r+1}\] 
for every $k \in \zz$ and $r \geq r_{0}$.

\begin{thm} [{\cite[Proof of Proposition 4.1]{cmnr-range}}] \label{eski-range}
 Let $\left(V^{\star}_{r} : r \geq r_{0}\right)$ be a cohomological single-graded spectral sequence of $\FI$-modules such that at the inital page $r_{0}$, for every $k \in \zz$ the $\FI$-module $V^{k}_{r_{0}}$ satisfies
\[
 \weak\!\left(V_{r_{0}}^{k}\right) < \infty \quad \text{and} \quad
 \local\!\left(V_{r_{0}}^{k}\right) < \infty \, .
\] 
Then for every $r \geq r_{0}$ and $k \in \zz$, the following hold:
\begin{birki}
 \item $\weak\!\left(V_{r}^{k}\right) \leq \weak\!\left(V_{r_{0}}^{k}\right)$.
 \vspace{0.1cm}
 \item $\local\!\left(V_{r}^{k}\right) \leq 
 \max\!\left( 
\begin{array}{l}
\left\{
 \local\!\left(V_{r_{0}}^{\ell}\right) : \ell \leq k+r-r_{0}
 \right\} \cup
 \vspace{0.09cm}
 \\ 
 \left\{ 2 \weak\!\left(V_{r_{0}}^{\ell} \right) - 2 : \ell \leq k + r - r_{0} - 1\right\} 
\end{array}
\right)$.
\end{birki}
\end{thm}

We first note that in fact the invariants $h^{j}$ when $j \geq 2$ can be entirely controlled by the stable degree $\weak$, so that one can restrict attention to the behavior of $h^{0}$ and $h^{1}$ in the setting of Theorem \ref{eski-range}.

\begin{thm}[{\cite[Theorem 2.10, part (4)]{cmnr-range}}] \label{after-2}
 Let $V$ be an $\FI$-module $V$ with $\weak(V) < \infty$ and $\local(V) < \infty$. Then for every $\pmb{j \geq 2}$, we have 
 \[
 h^{j}(V) \leq \max\!\left\{-1,\,2\weak(V) - 2j + 2\right\} \, .
 \]
\end{thm}
 We can now state our main result.
\begin{thmx} \label{yeni-range}
 Assume the setup and notation of Theorem \ref{eski-range}. Then for every $r \geq r_{0}$ and $k \in \zz$, setting
 \[
  C_{k} \coloneqq \max\!\left\{
  h^{0}\!\left(V_{r_{0}}^{k+1}\right),\,
  h^{1}\!\left(V_{r_{0}}^{k}\right),\,
  2\weak\!\left(V_{r_{0}}^{k} \right)-2,\,
  2\weak\!\left(V_{r_{0}}^{k-1} \right)-2
  \right\} \, ,
 \]
we have
\[
  h^{0}\!\left(V_{r}^{k}\right) \leq C_{k-1} \quad \text{and} 
  \quad
  h^{1}\!\left(V_{r}^{k}\right) \leq C_{k} \, .
\]
\end{thmx}

In applications of $\FI$-modules, invariants other than $\weak$ and $\local$  are often needed. We proceed with their introduction so that we can state a comprehensive result involving more invariants.

\paragraph{$\FI$-homology and regularity.} \label{defn:reg}
Consider the functor $\cofi{0} : \lMod{\FI} \rarr \lMod{\FI}$ defined via 
\begin{align*}
 \cofi{0}(V)_{S} := \coker\! \left(
 \bigoplus_{T \subsetneq S} V_{T} \rarr V_{S}
 \right)
\end{align*}
for every finite set $S$, which is right exact. For each $i \geq 0$, we write $\cofi{i} := \operatorname{L}_{i}\!\cofi{0}$ for its $i$-th left derived functor, and write 
\begin{align*}
 t_{i}(V) &:= \deg\!\left( \cofi{i}(V) \right) 
 \\
 &\in \{-1,0,1,\dots\} \cup \{\infty\} \, ,
 \\
 \reg(V) &:= \max\{t_{i}(V) - i : i \geq 1\} \\
  &\in\{-2,-1,0,1,\dots\} \cup\{\infty\} \, ,
\end{align*}
for every $\FI$-module $V$. We say that $V$ is \textbf{presented in finite degrees} if $\tgen(V)$ and $\trel(V)$ are both finite. We call $\reg(V)$ the \textbf{regularity} of $V$.

\begin{rem}
 An $\FI$-module $V$ is presented in finite degrees if and only if $\weak(V) < \infty$ and $\local(V) < \infty$ by \cite[Proposition 3.1]{cmnr-range} and \cite[Theorem 2.11]{bahran-polynomial}.
\end{rem}

Thanks to the known relationships between the various invariants of $\FI$-modules we deduce from Theorem \ref{yeni-range} the following.

\begin{thmx} \label{komple}
 Let $\left(V^{\star}_{r} : r \geq r_{0}\right)$ be a cohomological single-graded spectral sequence of $\FI$-modules and 
 \[
 (\weak_{k} : k \in \zz),\, \quad
 (\alpha_{k} : k \in \zz),\, \quad
 (\beta_{k} : k \in \zz),\, \quad
 \]
 three $\zz$-indexed sequences of integers $\geq -1$ such that at the inital page $r_{0}$, for every $k \in \zz$ the $\FI$-module $V^{k}_{r_{0}}$ is presented in finite degrees with
 \[
  \weak\!\left( V_{r_{0}}^{k}\right) \leq \weak_{k} \, , \quad 
  \quad
  h^{0}\!\left( V_{r_{0}}^{k}\right) \leq \alpha_{k} \, , \quad 
  \quad
  h^{1}\!\left( V_{r_{0}}^{k}\right) \leq \beta_{k}  \, .
 \] 
Then for every $r \geq r_{0}$ and $k \in \zz$, the $\FI$-module $V_{r}^{k}$ at page $r$ satisfies the following:
\begin{birki}
 \item $\weak\!\left(V_{r}^{k}\right) \leq \weak_{k}$.
 \vspace{0.1cm}
 \item For every $j \in \zz_{\geq 0}$, we have
 \[
 h^{j}\!\left(V_{r}^{k}\right) \leq 
\begin{cases}
 \max\!\left\{
  \alpha_{k},\,
  \beta_{k-1},\,
  2\weak_{k-1}-2,\,
  2\weak_{k-2}-2
  \right\} & \text{if $j=0$,}
  \vspace{0.1cm}
  \\
  \max\!\left\{
  \alpha_{k+1},\,
  \beta_{k},\,
  2\weak_{k}-2,\,
  2\weak_{k-1}-2
  \right\} & \text{if $j=1$,}
  \vspace{0.1cm}
  \\
  \max\!
  \left\{-1,\,
  2\weak_{k} -2j+2
  \right\} & \text{if $j \geq 2$.}
\end{cases}
 \]
 \vspace{0.1cm}
 \item $\local\!\left(V_{r}^{k}\right) \leq \max\!\left(
\begin{array}{l}
 \left\{
 \alpha_{k},\, 
 \alpha_{k+1}
 \right\} \cup
 \left\{
 \beta_{k-1},\, 
 \beta_{k}
 \right\} \cup
 \vspace{0.1cm}
 \\
 \left\{ 2\weak_{k}-2,\,
 2\weak_{k-1}-2,\,
 2\weak_{k-2}-2
 \right\}
\end{array}
 \right)$.
  \vspace{0.18cm}
 \item Both $\tgen\!\left(V_{r}^{k}\right)$ and $\reg\!\left(V_{r}^{k}\right)$ (and hence $\trel\!\left(V_{r}^{k}\right)-1$) are bounded above by
 \[
 \begin{cases}
  \alpha_{k} & \text{if $\weak_{k} = -1$,}
  \\
  \max\!\left(
\begin{array}{l}
 \left\{
 \alpha_{k},\, 
 \alpha_{k+1} + 1
 \right\} \cup
 \left\{
 \beta_{k-1},\, 
 \beta_{k} + 1
 \right\} \cup
 \vspace{0.1cm}
 \\
 \left\{ 2\weak_{k},\,
 2\weak_{k-1}-1,\,
 2\weak_{k-2}-2
 \right\}
\end{array}
 \right) & \text{if $\weak_{k} \geq 0$.}
\end{cases}
\]
\end{birki}
\end{thmx}

Finally, we state a corollary to Theorem \ref{komple} which roughly reproduces the bounds at the initial page of a spectral sequence at every further page with no losses.

\begin{corx}
 Let $\left(V^{\star}_{r} : r \geq r_{0}\right)$ be a cohomological single-graded spectral sequence of $\FI$-modules and $(\weak_{k} : k \in \zz)$ a $\zz$-indexed sequence of integers $\geq -1$ such that at the inital page $r_{0}$, for every $k \in \zz$ the $\FI$-module $V^{k}_{r_{0}}$ is presented in finite degrees with
\begin{itemize}
 \item $\weak\!\left( V_{r_{0}}^{k}\right) \leq \weak_{k}$,
 \vspace{0.1cm}
 \item $h^{0}\!\left( V_{r_{0}}^{k}\right) \leq \max\!\left\{
  -1,\,
  2\weak_{k-1}-2,\,
  2\weak_{k-2}-2
  \right\}$,
 \vspace{0.1cm}
 \item $h^{1}\!\left( V_{r_{0}}^{k}\right) \leq \max\!\left\{
  -1,\,
  2\weak_{k}-2,\,
  2\weak_{k-1}-2
  \right\}$.
\end{itemize}
Then for every $r \geq r_{0}$ and $k \in \zz$, the $\FI$-module $V_{r}^{k}$ at page $r$ satisfies the following:
\begin{birki}
 \item $\weak\!\left(V_{r}^{k}\right) \leq \weak_{k}$.
 \vspace{0.1cm}
 \item For every $j \in \zz_{\geq 0}$, we have
 \[
 h^{j}\!\left(V_{r}^{k}\right) \leq 
\begin{cases}
 \max\!\left\{
  -1,\,
  2\weak_{k-1}-2,\,
  2\weak_{k-2}-2
  \right\} & \text{if $j=0$,}
  \vspace{0.1cm}
  \\
  \max\!\left\{
  -1,\,
  2\weak_{k}-2,\,
  2\weak_{k-1}-2
  \right\} & \text{if $j=1$,}
  \vspace{0.1cm}
  \\
  \max\!
  \left\{-1,\,
  2\weak_{k} -2j+2
  \right\} & \text{if $j \geq 2$.}
\end{cases}
 \]
 \vspace{0.1cm}
 \item $\local\!\left(V_{r}^{k}\right) \leq \max\!
 \left\{ -1,\, 2\weak_{k}-2,\,
 2\weak_{k-1}-2,\,
 2\weak_{k-2}-2
 \right\}$.
  \vspace{0.18cm}
 \item Both $\tgen\!\left(V_{r}^{k}\right)$ and $\reg\!\left(V_{r}^{k}\right)$ (and hence $\trel\!\left(V_{r}^{k}\right)-1$) are bounded above by
 \[
 \begin{cases}
  \max\!\left\{
  -1,\,
  2\weak_{k-1}-2,\,
  2\weak_{k-2}-2
  \right\} & \text{if $\weak_{k} = -1$,}
  \\
  \max\!
 \left\{ 2\weak_{k},\,
 2\weak_{k-1}-1,\,
 2\weak_{k-2}-2
 \right\}
  & \text{if $\weak_{k} \geq 0$.}
\end{cases}
\]
\end{birki}
\end{corx}
\begin{proof}
 We may apply Theorem \ref{komple} with the same $\weak_{k}$ and 
\begin{align*}
 \alpha_{k} &\coloneqq \max\!\left\{
  -1,\,
  2\weak_{k-1}-2,\,
  2\weak_{k-2}-2
  \right\} \, ,
  \\
 \beta_{k} &\coloneqq \max\!\left\{
  -1,\,
  2\weak_{k}-2,\,
  2\weak_{k-1}-2
  \right\} \, ,
\end{align*}
for every $k \in \zz$. Then for every $k \in \zz$ we have $\alpha_{k} = \beta_{k-1}$ so that
\begin{align*}
 \max\!\left\{
  \alpha_{k},\,
  \beta_{k-1}
  \right\} &= \alpha_{k} 
  \\
  &= \max\!\left\{
  -1,\,
  2\weak_{k-1}-2,\,
  2\weak_{k-2}-2
  \right\} \, ,
  \\
  \max\!\left\{\alpha_{k},\,\alpha_{k+1}\right\} 
  &=
  \max\!\left\{\beta_{k-1},\,\beta_{k}\right\} 
  \\
  &= \{-1,\,2\weak_{k}-2,\,2\weak_{k-1}-2,\, 2\weak_{k-2}-2\} \, ,
  \\
  \max\!\left\{\alpha_{k},\,\alpha_{k+1} + 1\right\} 
  &=
  \max\!\left\{\beta_{k-1},\,\beta_{k} + 1\right\} 
  \\
  &= \{0,\,
  2\weak_{k}-1,\,
  2\weak_{k-1}-1,\,
  2\weak_{k-2}-2\} \, ,
\end{align*}
yielding all the bounds as stated.
\end{proof}
\section{Preliminaries} 

\begin{prop}[{\cite{cmnr-range}}] \label{ker-coker}
 Let $f \colon A \to B$ be a map of $\FI$-modules presented in finite degrees. Then the $\FI$-modules $\ker f$, $\coker f$ are also presented in finite degrees and the following hold:
\begin{birki}
 \item $\weak(\ker f) \leq \weak(A)$.
 \vspace{0.1cm}
 \item $\weak(\coker f) \leq \weak(B)$.
 \vspace{0.1cm}
 \item $h^{0}(\ker f) \leq h^{0}(A)$.
 \vspace{0.1cm}
 \item $h^{0}(\coker f) \leq \max\{h^{0}(B),\, h^{1}(A),\,2\weak(A) - 2\}$.
 \vspace{0.1cm}
 \item $h^{1}(\ker f) \leq \max\{h^{0}(B),\, h^{1}(A)\}$.
 \vspace{0.1cm}
 \item $h^{1}(\coker f) \leq \max\{h^{1}(B),\, 2\weak(A)-2\}$. 
\end{birki}
\end{prop}
\begin{proof}
 Both $\ker f$ and $\coker f$ are presented in finite degrees by \cite[Theorem 2.3]{cmnr-range}. The bounds (1) and (2) are literally \cite[Proposition 3.3, (1) and (2)]{cmnr-range}. The bounds (3) and (5) are obtained in \cite[proof of Proposition 3.3, page 18]{cmnr-range}. At the end of the same proof, it is also obtained that
 \[
  h^{j}(\coker f) \leq \max\{h^{j}(B),\, h^{j+1}(A),\, h^{j+2}(\ker f)\} 
 \]
 for every $j \geq 0$. Setting $j \coloneqq 0$ here and using Theorem \ref{after-2} together with part (1) yields
\begin{align*}
 h^{0}(\coker f) &\leq \max\{h^{0}(B),\, h^{1}(A),\, h^{2}(\ker f)\} 
 \\
 &\leq \max\{h^{0}(B),\, h^{1}(A),\, 2\weak(\ker f) - 2\} 
 \\
 &\leq \max\{h^{0}(B),\, h^{1}(A),\, 2\weak(A) - 2\} 
\end{align*}
hence the bound (4) follows, whereas setting $j \coloneqq 1$ and using Theorem \ref{after-2} together with part (1) yields
\begin{align*}
 h^{1}(\coker f) &\leq \max\{h^{1}(B),\, h^{2}(A),\, h^{3}(\ker f)\} 
 \\
 &\leq \max\{h^{1}(B),\, 2\weak(A)-2,\, 2\weak(\ker f) - 4\} 
 \\
 &\leq \max\{h^{1}(B),\, 2\weak(A)-2,\, 2\weak(A) - 4\} 
 = \max\{h^{1}(B),\, 2\weak(A)-2\}
\end{align*}
hence the bound (6) follows.
\end{proof}

\begin{cor} \label{homology-bounds}
 Let $X \xrightarrow{f} E \xrightarrow{g} Y$ be maps of $\FI$-modules presented in finite degrees with $gf = 0$ and set
 \[
  H \coloneqq \ker g / \im f \, .
 \]
  Then the $\FI$-module $H$ is also presented in finite degrees and the following hold:
\begin{birki}
 \item $\weak(H) \leq \weak(E)$.
 \vspace{0.1cm}
 \item $h^{0}(H) \leq \max\{h^{0}(E),\, h^{1}(X),\,2\weak(X) - 2\}$.
 \vspace{0.1cm}
 \item $h^{1}(H) \leq \max\{h^{0}(Y),\, h^{1}(E),\, 2\weak(X)-2\}$.
\end{birki}
\end{cor}
\begin{proof}
 Since $H \cong \coker\!\left(X \xrightarrow{f} \ker g\right)$, by Proposition \ref{ker-coker} the $\FI$-module $H$ is presented in finite degrees with 
\begin{align*}
 \weak(H) &\leq \weak\!\left(
 \ker g
 \right) \leq \weak(E) \, ,
 \\
 h^{0}(H) &\leq \max\{h^{0}(\ker g),\, h^{1}(X),\,2\weak(X) - 2\}
 \leq \max\{h^{0}(E),\, h^{1}(X),\,2\weak(X) - 2\}\,,
 \\
 h^{1}(H) &\leq \max\{h^{1}(\ker g),\, 2\weak(X)-2\} \leq 
 \max\{h^{0}(Y),\, h^{1}(E),\, 2\weak(X)-2\}\,.
\end{align*}
\end{proof}

\section{Proofs}
\begin{proof}[Proof of \textbf{\emph{Theorem \ref{yeni-range}}}]
 We employ induction on $r$. The base case $r=r_{0}$ follows immediately from the definition of $C_{k}$'s. Next, assume the claim holds for a fixed $r \geq r_{0}$. At page $r+1$, by definition for every $k \in \zz$ the $\FI$-module $V^{k}_{r+1}$ is the homology of the complex
 \[ 
  V_{r}^{k-1} \to V_{r}^{k} \to V_{r}^{k+1} \, ,
 \]
so by part (1) of Theorem \ref{eski-range}, Corollary \ref{homology-bounds} and the induction hypothesis, we have
\begin{align*}
 h^{0}\!\left(V_{r+1}^{k}\right) 
 &\leq \max\!\left\{h^{0}\!\left( V_{r}^{k} \right),\, 
 h^{1}\!\left(V_{r}^{k-1}\right),\,
 2\weak\!\left( V_{r}^{k-1} \right) - 2\right\}
 \\
 &\leq \max\!\left\{C_{k-1},\, 
 2\weak\!\left( V_{r_{0}}^{k-1} \right) - 2\right\} = C_{k-1}
\end{align*}
and
\begin{align*}
 h^{1}\!\left( V^{k}_{r+1} \right) &\leq 
 \max\!\left\{
 h^{0}\!\left( V_{r}^{k+1} \right),\, 
 h^{1}\!\left( V_{r}^{k} \right),\, 
 2\weak\!\left( V_{r}^{k-1} \right)-2
 \right\}
 \\
 &\leq \max\!\left\{ C_{k},\,
 2\weak\!\left( V_{r_{0}}^{k-1} \right) - 2  \right\} = C_{k} \, ,
\end{align*}
establishing the claim for the page $r+1$.
\end{proof}

\begin{proof}[Proof of \emph{\textbf{Theorem \ref{komple}}}]
 At the initial page $r_{0}$, for every $k \in \zz$ the 
the $\FI$-module $V^{k}_{r_{0}}$ satisfies
\[
 \weak\!\left(V_{r_{0}}^{k}\right) < \infty \quad \text{and} \quad
 \local\!\left(V_{r_{0}}^{k}\right) < \infty 
\] 
by \cite[Proposition 3.1]{cmnr-range}.

The bound (1) follows from part (1) of Theorem \ref{eski-range} by induction on $r$. Setting
\[
  C_{k} \coloneqq \max\!\left\{
  h^{0}\!\left(V_{r_{0}}^{k+1}\right),\,
  h^{1}\!\left(V_{r_{0}}^{k}\right),\,
  2\weak\!\left(V_{r_{0}}^{k} \right)-2,\,
  2\weak\!\left(V_{r_{0}}^{k-1} \right)-2
  \right\} \, ,
 \] 
by the hypotheses and part (1) we have
\[
 C_{k} \leq \max\!\left\{
  \alpha_{k+1},\,
  \beta_{k},\,
  2\weak_{k}-2,\,
  2\weak_{k-1}-2
  \right\} \, .
\]
Now Theorem \ref{yeni-range} yields
\begin{align*}
  h^{0}\!\left(V_{r}^{k}\right) 
 &\leq C_{k-1} 
 \leq
 \max\!\left\{
  \alpha_{k},\,
  \beta_{k-1},\,
  2\weak_{k-1}-2,\,
  2\weak_{k-2}-2
  \right\} \, , 
 \\
 h^{1}\!\left(V_{r}^{k}\right) 
 &\leq C_{k}
 \leq \max\!\left\{
  \alpha_{k+1},\,
  \beta_{k},\,
  2\weak_{k}-2,\,
  2\weak_{k-1}-2
  \right\} \, .
\end{align*}
Together with part (1) and Theorem \ref{after-2}, the bounds in (2) follow. The bound (3) follows immediately from (2) and the definition of $\local$. 

If $\weak_{k} = -1$, then by part (1) and the definition of $\weak$, the $\FI$-module $V_{r_{0}}^{k}$ is torsion and hence has finite degree; thus its subquotient $V_{r}^{k}$ also has finite degree so that by \cite[Proposition 4.9]{ramos-coh} we have
\begin{align*}
 h^{j}\!\left(V_{r}^{k}\right) &\leq 
\begin{cases}
 \deg\!\left(V_{r_{0}}^{k}\right) & \text{if $j=0$,}
 \\
 -1 & \text{if $j \geq 1$.}
\end{cases}
\\ 
 &\leq 
\begin{cases}
 \alpha_{k} & \text{if $j=0$,}
 \\
 -1 & \text{if $j \geq 1$.}
\end{cases}
\end{align*}
Let us write
\[
 R_{k} \coloneqq \max\!\left(
\begin{array}{l}
 \left\{
 \alpha_{k},\, 
 \alpha_{k+1} + 1
 \right\} \cup
 \left\{
 \beta_{k-1},\, 
 \beta_{k} + 1
 \right\} \cup
 \vspace{0.1cm}
 \\
 \left\{ 2\weak_{k},\,
 2\weak_{k-1}-1,\,
 2\weak_{k-2}-2
 \right\}
\end{array}
 \right) \, .
\]
In general, by \cite[Corollary 4.15]{ramos-coh} and part (2), we have
\begin{align*}
 \reg\!\left(V_{r}^{k}\right) 
 &\leq \max\!\left( \{-1\} \cup\left\{
 h^{j}\!\left(V_{r}^{k}\right) + j : h^{j}\!\left(V_{r}^{k}\right) \geq 0
 \right\}\right)
 \\
 &\leq
\begin{cases}
 \alpha_{k} & \text{if $\weak_{k} = -1$,}
 \\
 \max\!\left\{
 h^{j}\!\left(V_{r}^{k}\right) + j : 0 \leq j \leq \weak_{k}+1
 \right\}
 & \text{if $\weak_{k} \geq 0$.} 
\end{cases}
 \\
 &\leq
\begin{cases}
 \alpha_{k} & \text{if $\weak_{k} = -1$,}
 \\
 \max\!\left( 
\begin{array}{l}
 \left\{h^{0}\!\left(V_{r}^{k}\right),\,
 h^{1}\!\left(V_{r}^{k}\right) + 1\right\} \cup
 \\
 \left\{
 2\weak_{k} - j + 2 : 2 \leq j \leq \weak_{k}+1
 \right\} 
\end{array}
 \right)
 & \text{if $\weak_{k} \geq 0$.} 
\end{cases} 
 \\
 &\leq
\begin{cases}
 \alpha_{k} & \text{if $\weak_{k} = -1$,}
 \\
 \max\!
 \left\{h^{0}\!\left(V_{r}^{k}\right),\,
 h^{1}\!\left(V_{r}^{k}\right) + 1,\, 2\weak_{k}\right\}
 & \text{if $\weak_{k} \geq 0$.} 
\end{cases} 
 \\
 &\leq
\begin{cases}
 \alpha_{k} & \text{if $\weak_{k} = -1$,}
 \\
 R_{k}
 & \text{if $\weak_{k} \geq 0$.} 
\end{cases} 
\end{align*}
 Finally by \cite[Corollary 2.10]{bahran-reg}, part (1), and the regularity bound above, we also have
\begin{align*}
 \tgen\!\left( V_{r}^{k} \right) 
 &\leq
  \max\!\left\{\weak\!\left( V_{r}^{k} \right),\,
  \trel\!\left( V_{r}^{k} \right) - 1 \right\}
  \\
 &\leq
  \max\!\left\{\weak_{k},\,
  \reg\!\left( V_{r}^{k} \right) \right\} 
  \\
 &\leq
\begin{cases}
 \max\{-1,\,\alpha_{k}\} & \text{if $\weak_{k} = -1$,}
 \\
 \max\{\weak_{k},\,R_{k}\}
 & \text{if $\weak_{k} \geq 0$.} 
\end{cases} 
 \\
 &\leq
\begin{cases}
 \alpha_{k} & \text{if $\weak_{k} = -1$,}
 \\
 R_{k}
 & \text{if $\weak_{k} \geq 0$.}  
\end{cases}
\end{align*}
because $\alpha_{k} \geq -1$ and $R_{k} \geq 2\weak_{k} \geq \weak_{k}$ when $\weak_{k} \geq 0$.
\end{proof}

\bibliographystyle{hamsalpha}
\bibliography{stable-boy}

\end{document}